\documentclass[art]{amsart}

\usepackage{amssymb,latexsym}
\usepackage{amsfonts}
\usepackage{multirow}
\usepackage{dsfont}
\usepackage[pdftex]{hyperref}
\usepackage{color}
\usepackage{graphicx}
\usepackage[all]{xy}
\usepackage{ytableau}
\usepackage[vcentermath]{youngtab}

\theoremstyle{plain}
\newtheorem{thm}[equation]{Theorem}
\newtheorem{lem}[equation]{Lemma}

\newtheorem{question}[equation]{Question}

\newtheorem{prop}[equation]{Proposition}

\newtheorem{example}[equation]{Example}
\newtheorem{THM}{Theorem}

\newtheorem{cor}[equation]{Corollary}
\newtheorem{rem}[equation]{Remark}
\newtheorem{defn}[equation]{Definition}
\numberwithin{equation}{section}

\newcommand{\Z}{\mathbb Z}

\newcommand{\C}{\mathbb C}
\DeclareMathOperator{\Lie}{Lie}
\newcommand{\fg}{\mathfrak{g}}

\newcommand{\fb}{\mathfrak{b}}

\newcommand{\fm}{\mathfrak{m}}
\newcommand{\ft}{\mathfrak{t}}

\newcommand{\fu}{\mathfrak{u}}

\newcommand{\fp}{\mathfrak{p}}

\newcommand{\B}{\mathcal{B}}
\newcommand{\cx}{\mathcal{X}}

\newcommand{\RST}{{\tt RST}}
\newcommand{\RSST}{{\tt RSST}}
\newcommand{\SST}{{\tt SST}}

\begin{document}

\title[Hessenberg varieties of parabolic type]{Hessenberg varieties of parabolic type}

\author{Martha Precup}
\address{Department of Mathematics and Statistics\\ Washington University in St. Louis \\ One Brookings Drive \\ St. Louis, Missouri  63130 \\ U.S.A. }
\email{martha.precup@wustl.edu}

\author{Julianna Tymoczko}
\address{Dept. of Mathematics, Smith College, Northampton, Massachusetts  01063}\email{jtymoczko@smith.edu}


%
%

%
\maketitle

 \begin{abstract} This paper studies the geometry and combinatorics of three interrelated varieties: Springer fibers, Steinberg varieties, and parabolic Hessenberg varieties.  We prove that each parabolic Hessenberg variety is the pullback of a Steinberg variety under the projection of the flag variety to an appropriate partial flag variety and we give three applications of this result.  The first application constructs an explicit paving of all Steinberg varieties in Lie type $A$ in terms of semistandard tableaux. As a result, we obtain an elementary proof of a theorem of Steinberg and Shimomura that the well-known Kostka numbers count the maximal-dimensional irreducible components of Steinberg varieties.   The second application proves an open conjecture for certain parabolic Hessenberg varieties in Lie type A by showing that their Betti numbers equal those of a specific union of Schubert varieties. The third application proves that the irreducible components of parabolic Hessenberg varieties are in bijection with the irreducible components of the Steinberg variety.  All three of these applications extend our geometric understanding of the three varieties at the heart of this paper, a full understanding of which is unknown even for Springer varieties, despite over forty years' worth of work.  
\end{abstract}

 \section{Introduction}
 
 In this paper, we study the geometric and combinatorial structure of three interrelated varieties, using properties of one variety to infer new information about the others.  We now introduce these varieties in Lie type $A$ though much of the paper treats arbitrary Lie type.  Two of these varieties are subvarieties of the flag variety $G/B$, which in type $A$ is the collection of nested complex vector spaces $V_1 \subseteq V_2 \subseteq \cdots \subseteq V_{n-1} \subseteq \mathbb{C}^n$ where each $V_i$ is $i$-dimensional.  The third is a subvariety of the partial flag variety $G/P$, which in type $A$ is a family that includes the Grassmannian $G(k,n)$ of $k$-dimensional subspaces of a fixed $\mathbb{C}^n$.  The three main objects we consider are the following.
 
 \begin{enumerate}
 \item {\bf Springer fibers:}  Defined by a nilpotent linear operator $X$, the Springer fiber $\B^X$ is the family of flags that are fixed by $X$ in the sense that $XV_i \subseteq V_i$ for all $i$.  Springer proved that  the cohomology of the Springer fibers carries an action of $S_n$ in what is often considered a first example of a geometric representation theory \cite{Sp, Sp2}. The geometry of Springer fibers is deeply connected to the combinatorics of permutations and $S_n$-representations.  However, little is known about Springer fibers for general $X$ except the Betti numbers \cite{Fr, T} and that they are are pure dimensional with components indexed by standard tableaux \cite{S}.  More is known about the components themselves for particular $X$, e.g. if $X^2=0$ \cite{FM}, the Jordan type of $X$ has two blocks \cite{Fr2, Fu, Wilbert, ILW}, or when the irreducible components of $B^X$ are smooth \cite{GZ}.

 \item {\bf Parabolic Hessenberg varieties:} Hessenberg varieties loosen the condition used to define Springer fibers.  Given a linear operator $X$ and a nondecreasing function $h: \{1,2,\ldots,n\} \rightarrow \{1,2,\ldots,n\}$ the Hessenberg variety $\B(X,h)$ consists of the flags that $X$ moves by no more than $h$, in the sense that $XV_i \subseteq V_{h(i)}$ for all $i$.  Motivated by Hessenberg matrices and algorithms for efficiently calculating eigenvalues in numerical analysis, Hessenberg varieties in the flag variety of $GL_n(\C)$ were first introduced by De Mari and Shayman \cite{DS1988} and later defined in all Lie types by De Mari, Procesi, and Shayman \cite{dMPS}. Independently, Peterson and Kostant used them to construct the quantum cohomology of the flag variety \cite{Ko2} (see also~\cite{Rietsch2003}).  When $X$ has $n$ distinct eigenvalues, the equivariant cohomology of the corresponding Hessenberg variety carries an $S_n$-action \cite{T2} that can be described by certain quasisymmetric functions (see the conjecture by Shareshian and Wachs \cite{SW} and recent proof from Brosnan and Chow \cite{BC} and independently Guay-Paquet \cite{G}).  As with Springer fibers, this endows the Betti numbers of Hessenberg varieties with combinatorial and representation-theoretic significance.  Many people have analyzed these Betti numbers and cohomology rings for special cases of $X$ and $h$ (see \cite{T, Precup2018, M, AHHM, AHMMS2016} for just a few examples), though as with Springer fibers, the general geometric structure of Hessenberg varieties remains mysterious.
 
 This paper considers the case when $h$ corresponds to a parabolic subalgebra, which occurs when the image of $h$ consists of precisely those $i$ that are fixed by $h$.  (If $i_1 < i_2$ are two consecutive fixed points of $h$ then $h(i_1+1)=h(i_1+2)=\cdots=h(i_2)=i_2$.  This means $h$ describes the column-heights of a block-upper-triangular collection of matrices, namely a parabolic subalgebra of the $n \times n$ matrices.)  
 
 \item {\bf Steinberg varieties:} Steinberg varieties loosen the condition used to define Springer fibers in a different way.  Given a linear operator $X$ and an integer $1 \leq k < n$ the Steinberg variety associated to $X$ and $k$ is the collection of $k$-planes $V_k$ with $XV_k \subseteq V_k$.  More generally, if $X$ is a linear operator and $J$ is the index set of any partial flag variety $G/P_J$ with elements $V_{i_1} \subseteq V_{i_2} \subseteq \cdots \subseteq \mathbb{C}^n$ then the Steinberg variety corresponding to $X$ and $J$ is the image $\pi_J(\B^X)$ under the standard projection $\pi_J: G/B \rightarrow G/P_J$ obtained by forgetting subspaces not indexed by $i \in J$.  (We denote Steinberg varieties thus throughout this paper.) Steinberg proved that the irreducible components of $\pi_J(\B^X)$ of maximal dimension are counted by the {\em Kostka numbers,} a well-known quantity in algebraic combinatorics \cite{Steinberg}.  Borho and MacPherson computed the cohomology of the Steinberg variety $\pi_J(\B^X)$, identifying it with the subspace of $W_J$-invariants of the Springer representation on $H^*(\B^X)$ where $W_J$ is generated by the simple reflections $s_i$ for $i\notin J$ \cite{Borho-MacPherson}.  More recently, Fresse proved all Steinberg varieties are paved by affines~\cite{Fresse}.  Little else is known about the geometry of Steinberg varieties.  
 \end{enumerate}

This paper analyzes the topological structure of parabolic Hessenberg varieties.  Our main result  proves that each parabolic Hessenberg variety is the pull-back of a Steinberg variety under the projection to a partial flag variety  (c.f.~Theorem~\ref{pullback} below.)

\begin{THM}\label{Thm0} Let $h:\{1,2,\ldots,n\} \rightarrow \{1,2,\ldots,n\}$ be a parabolic Hessenberg function with fixed points $J=\{i_1,i_2,\ldots,i_k\}$ and let $\pi_J: G/B \to G/P_J$ be the corresponding projection of the full flag variety to the partial flag variety obtained by forgetting subspaces $V_i$ with $i \not \in J$.  The parabolic Hessenberg variety $\B(X,h)$ is the pull-back of the Steinberg variety $\pi_J(\B^X)$ under $\pi_J$. 
\end{THM}
 
We use this theorem to give an explicit formula for the Poincar\'e polynomial of a parabolic Hessenberg variety for those $X$ that satisfy the assumptions of Theorem~\ref{thm.paving}. Theorem~\ref{Poincare polynomial} proves it is the product of the Poincar\'e polynomial of the Steinberg variety and Poincar\'e polynomial of a smaller flag variety.   As a corollary, we show that the Poincar{\'e} polynomial of a parabolic Hessenberg variety is the shifted sum of the Poincar\'{e} polynomial of the Steinberg variety, with shifts determined by  $h$. 

Moreover our results explicitly lay out the combinatorics of a paving for both Steinberg varieties and parabolic Hessenberg varieties when $X$ satisfies the assumptions of Theorem~\ref{thm.paving}.  This allows us to specify Betti numbers for Steinberg and parabolic Hessenberg varieties, and to recover Fresse's proof that pavings of Steinberg varieties exist by explicitly producing a paving for these $X$.

We give three main applications of these results.

First, we develop an explicit combinatorial description of the paving of Steinberg varieties in type $A$ in terms of certain semistandard tableaux.  We recover a theorem of Steinberg  \cite{Steinberg} and Shimomura  \cite{Shimomura1, Shimomura2} that computes the number of irreducible and maximal-dimensional components of a Steinberg variety in terms of the well-known Kostka numbers.  However, our proof is more streamlined, grounded in the combinatorics of semistandard (versus standard) tableaux.

Second, we show that the Betti numbers of parabolic Hessenberg varieties for three-row or two-column nilpotent operators are equal to the Betti numbers of a specific union of Schubert varieties.  Schubert varieties are the closures of cells in the best-known CW-decomposition of the flag variety; they induce a cohomology basis for the flag variety, and their combinatorics and geometry are deeply intwined (see, for example, the books \cite{BL, F}). Varieties whose Betti numbers are those of a union of Schubert varieties admit a particularly simple construction of equivariant cohomology, as proven by Harada and the second author \cite{HT} and applied to certain Hessenberg varieties \cite{HTSchubert}.  Conjecturally, this applies to all nilpotent Hessenberg varieties.  The conjecture was confirmed for Hessenberg varieties when $X$ has a single Jordan block by Mbirika \cite{M}, who computed the Betti numbers,  and Reiner, who recognized them as those of a Schubert variety called the Ding variety \cite{D, DMR}.  More recently, it was also proven for three-row or two-column Springer fibers by the authors of the current paper \cite{PT}. 

Third and last, we give a new analysis of the irreducible components of parabolic Hessenberg varieties in Section \ref{Conclusion}.   We prove that the irreducible components of parabolic Hessenberg varieties are in bijection with those of the corresponding Steinberg variety, and state some consequences in the type $A$ case.

This paper is structured as follows.  The second section covers background information and notation.  The third analyzes the structure of parabolic Hessenberg varieties.  All the results in Section \ref{parabolic}, including our main result, hold for Hessenberg varieties defined using any complex algebraic reductive group.   The rest of the paper contains applications of this result.  The fourth section specializes to the case $G=GL_n(\C)$ and describes a paving of Steinberg varieties obtained by intersecting with Schubert cells.  The fifth section then proves in type $A$ that the Betti numbers of parabolic Hessenberg varieties are equal to those of a specific union of Schubert varieties. An analogous result holds for Steinberg varieties, except that the union of Schubert varieties is taken in the partial flag variety (which makes a significant difference).  Finally, Section \ref{Conclusion} concludes by studying the irreducible components of parabolic Hessenberg varieties.

 {\bf Acknowledgements.}  The first author was partially supported by an AWM-NSF mentoring grant during this work.  The second author was partially supported  by National Science Foundation grants DMS-1248171 and DMS-1362855.

 
 \section{Preliminaries}\label{preliminaries}
 
This section establishes key definitions, as well as some results that restate past work in the form that is most useful in what follows.  We fix the following notation:
\begin{itemize}
 \item $G$ is a complex algebraic reductive group with Lie algebra $\fg$.
 \item  $B$ is a fixed Borel subgroup of $G$  with Lie algebra $\fb$.
 \item $\Phi$ is the root system of $\fg$.
 \item $U$ is the maximal unipotent subgroup of $B$  with Lie algebra $\fu$.
 \item  $T\subset B$ is a fixed maximal torus with Lie algebra $\ft$.
 \item $W=N_G(T)/T$ denotes the Weyl group.
 \item We fix a representative $w\in N_G(T)$ for each $w\in W$ and use the same letter for both.
 \item  $\Phi^+$, $\Phi^-$, and $\Delta$ are the positive, negative and simple roots associated to the previous data.
 \item Given $\gamma\in \Phi$ we write $\fg_{\gamma}$ for the root space in $\fg$ corresponding to $\gamma$ and fix a generating root vector $E_{\gamma}\in \fg_{\gamma}$.
 \item We denote by $s_{\gamma}$ the reflection in $W$ corresponding to $\gamma\in \Phi$ and write $s_{\alpha_i}=s_i$ when $\alpha_i\in \Delta$.
 \end{itemize}
 
In Section~\ref{parabolic} we specialize to the case when $G=GL_n(\C)$ is the group of $n\times n$ invertible matrices and $\fg=\mathfrak{gl}_n(\C)$ is the collection of $n\times n$ matrices.  This is also our main example throughout.  In this setting, $B$ is the subgroup of invertible upper-triangular matrices,  $T$ is the diagonal subgroup, and $W\cong S_n$ is the symmetric group on $n$ letters.  The positive roots in this case are 
 \[
 \Phi^+=\{ \alpha_i+ \alpha_{i+1} \cdots +\alpha_{j-1}\mid 1\leq i < j \leq n \}
 \] 
 where $\alpha_i=\epsilon_{i}-\epsilon_{i-1}$ and $\epsilon_i(X)=X_{ii}$ for all $X\in \mathfrak{gl}_n(\C)$.  Let $E_{ij}$ denote the elementary matrix with $1$ in the $(i,j)$-entry and $0$ in every other entry. The root vector corresponding to the root $\gamma= \alpha_i+\alpha_{i+1} \cdots + \alpha_{j-1}$  for each $1\leq i<j \leq n$ is $E_{\gamma}=E_{i j}$.  When working in the type $A$ setting we sometimes identify $(i,j)$ with the root $ \alpha_i+ \alpha_{i+1} \cdots +\alpha_{j-1}$.  
 
\begin{defn}
The \textbf{inversion set} of the Weyl group element $w$ is the set
\[
N(w) = \{ \gamma\in \Phi^+ \mid w(\gamma)\in \Phi^- \}
\]  
\end{defn}

This generalizes to arbitrary Lie type the classical definition of an inversion, where the pair $(i,j)$ is  an inversion of $w\in S_n$ if $i<j$ and $w(i)>w(j)$. If we identify $(i,j)$ with the root $\alpha_i+ \alpha_{i+1}+\cdots + \alpha_{j-1}\in \Phi^+$ then $(i,j)$ is an inversion of $w$ in the classical sense if and only if $\alpha_i + \alpha_{i+1}+\cdots +\alpha_{j-1}\in N(w)$.  Note that if $\ell(w)$ denotes the (Bruhat) length function on $W$ then $\ell(w)=|N(w)|$.

The projective variety $G/B$ is called the \textbf{flag variety}.  When $G=GL_n(\C)$ the flag variety can be identified with the set of full flags $V_\bullet = (V_1 \subseteq V_2 \subseteq \cdots \subseteq V_{n-1} \subseteq V)$ in a complex $n$-dimensional vector space $V$ as in the Introduction.  Hessenberg varieties are parametrized by two objects: a Hessenberg space $H \subseteq \fg$ and an element $X \in \fg$.

\begin{defn}  A linear subspace $H\subseteq \fg$ is a \textbf{Hessenberg space} if $\fb\subseteq H$ and $[\fb, H] \subseteq H$.
\end{defn}

The condition that $[\fb, H]\subseteq H$ implies that this subspace of $\mathfrak{g}$ can be written as 
\[
H = \ft \oplus \bigoplus_{\gamma \in \Phi_H} \fg_{\gamma}
 \]  
over an index set $\Phi_H\subseteq \Phi$ determined by (and determining) $H$.  Let $\Phi_H^- = \Phi_H\cap \Phi^-$ denote the negative roots in this index set. When $\fg=\mathfrak{gl}_n(\C)$, the set of indices $\Phi_H$ forms a ``staircase" shape, in the sense that if $(i,j)$ corresponds to a root in $\Phi_H$ then so do all $(k,j)$ with $1 \leq k \leq i$ and all $(i,k)$ with $j \leq k \leq n$.  In other words if matrices in $H$ are not identically zero in the entry $(i,j)$, then they can be nonzero in any entry above or to the right of $(i,j)$. 

Each Hessenberg space $H\subseteq \mathfrak{gl}_n (\C)$ is uniquely associated to a Hessenberg function $h: \{ 1, ..., n \} \to \{ 1,..., n \}$ by the rule that $h(i)$ equals the number of entries that are not identically zero in the $i$-th column of $H$.  This is precisely the map $h$ from the Introduction.  The condition that $h(i)\geq i$ is equivalent to the requirement that $\fb \subseteq H$ while the condition $h(i) \geq h(i-1)$ is equivalent to the requirement $[\fb,H] \subseteq H$.  

We remark that the condition $\fb\subseteq H$ is typically, but not logically, necessary.  It is in any case implied when $H$ is a parabolic subalgebra, which is  the main focus of this paper.

\begin{example} We give a Hessenberg function $h$ and the corresponding Hessenberg space $H$ when $n=5$.  The space of matrices $H$ is described by indicating where the zeroes must be in each matrix; the entries designated $*$ can be filled freely with any element of $\mathbb{C}$.
\[H = \begin{pmatrix} * & * & * & * & * \\ * & * & * & * & * \\ 0 & 0 & * & * & * \\ 0 & 0 & * & * & * \\ 0 & 0 & 0 & * & * \end{pmatrix} \hspace{0.5in}  \longleftrightarrow \hspace{0.5in} h(i) = \left\{ \begin{array}{ll} 2 & \textup{ if } i=1,2 \\
4 & \textup{ if } i=3 \\ 5 & \textup{ if } i = 4,5 \end{array} \right. \]
\end{example}

This paper focuses on a family of subvarieties of the flag variety called Hessenberg varieties.   

\begin{defn} Fix a Hessenberg space $H\subset \fg$ and an element $X\in \fg$.  The \textbf{Hessenberg variety} associated to $X$ and $H$ is the subvariety of the flag variety given by
\[
\B(X,H)=\{ gB\in G/B \mid g^{-1}\cdot X \in H  \}
\] 
where $g\cdot X := Ad(g)X=gXg^{-1}$.
\end{defn}

In this paper, we assume $X\in \fg$ is nilpotent, in which case we say that the corresponding variety $\B(X,H)$ is a \textbf{nilpotent Hessenberg variety}.  A key example is the case in which $H=\fb$ and $X\in \fg$ is nilpotent.  Then $\B(X,\fb)$ consists of all flags $gB$ such that $g^{-1}\cdot X\in \fb$ or equivalently $X\in g\cdot \fb$.  This is called the \textbf{Springer fiber} and is denoted by $\B^X$.

Hessenberg varieties have an affine paving, which is like a CW-complex structure but with less restrictive closure conditions.

\begin{defn}  A \textbf{paving} of an algebraic variety $Y$ is a filtration by closed subvarieties
\[
Y_0\subset Y_1 \subset \cdots \subset Y_i \subset \cdots \subset Y_d=Y.
\]	
A paving is \textbf{affine} if every $Y_i-Y_{i-1}$ is a finite disjoint union of affine spaces.  In this case, we say that these affine spaces \textbf{pave} $Y$.
\end{defn}

Like CW-complexes, affine pavings can be used to compute the Betti numbers of a variety.

\begin{rem}\label{betti} Let $Y$ be an algebraic variety with an affine paving and let $n_k$ denote the number of affine components of dimension $k$, or zero if $n_k$ is zero.  Then the compactly-supported cohomology groups of $Y$ are given by $H_c^{2k}(Y)= \Z^{n_k}$.  \textup{(}For more, see e.g.~\cite[19.1.1]{F2}.\textup{)}
\end{rem}

The Bruhat decomposition of the flag variety induces a well-known paving by affines  \cite[Section 2.6]{BL}. Decompose the flag variety as $G/B=\bigsqcup_{w\in W} C_w$ where $C_w=BwB/B$ is the \textbf{Schubert cell} indexed by $w\in W$ and the closure $\overline{C}_w$ is a \textbf{Schubert variety}.  The paving of $G/B$ given by 
\[
(G/B)_i=\bigsqcup_{\ell(w)=i} \overline{C}_w
\]  
is affine because $\overline{C}_w=\bigsqcup_{y\leq w} C_y$ where $\leq$ denotes the Bruhat order and because $C_w\cong \C^{\ell(w)}$ for each $w$.  

Calculating the Poincar\'e polynomial of a Schubert variety or a union of Schubert varieties is a  application of this combinatorial description.  
\begin{example}\label{Schubertex}  Let $G=GL_4(\C)$ and consider $w=s_3s_2s_1s_3$.  The set $\{v\in W  \mid v\leq w\}$ is the set of all possible subwords of $w$.  When $w=s_3s_2s_1s_3$ this set is
\[
\{ s_3s_2s_1s_3, s_2s_1s_3, s_3s_2s_3, s_3s_2s_1, s_3s_2, s_3s_1, s_2s_1, s_2s_3, s_1, s_2, s_3, e\}
\]
Therefore the Poincar\'e polynomial of $\overline{C}_w$ is $P(\overline{C}_w, t)= 1+3t+4t^2+3t^3+t^4$.
\end{example}

Intersecting the Hessenberg variety $\B(X,H)$ with certain choices of Schubert cells gives an affine paving of $\B(X,H)$.  We call these intersections \textbf{Hessenberg Schubert cells} (or \textbf{Springer Schubert cells} if the underlying Hessenberg variety is in fact a Springer fiber).  We now describe the Hessenberg Schubert cells that we use in this paper.  Note that $\B(X,H)$ and $\B(g\cdot X, H)$ are homemorphic (see, for example, the one-line proof in \cite[Proposition 2.7]{T}).

Let $X\in \fg$ be nilpotent and fix $H$.   The previous paragraph says that we can choose $X$ within its conjugacy class to make computations as convenient as possible. We now describe  one such choice when $\fg=\mathfrak{gl}_n(\C)$. This particular operator will play an important role in the combinatorial results of Sections~\ref{section: Betti numbers of Steinberg} and~\ref{Schubert points}.  Recall that the conjugacy classes of nilpotent matrices in $\mathfrak{gl}_n(\C)$ are determined  by the sizes of their Jordan blocks.  Let $\lambda$ be a partition of $n$.    We first construct a representative for the nilpotent conjugacy class of Jordan type $\lambda$ as in~\cite[\S 4]{T}.   

\begin{defn}\label{defn: highest form} Let $\lambda=(\lambda_1, \lambda_2, \ldots, \lambda_k)$ be a partition of $n$, drawn as a Young diagram with $\lambda_i$ boxes in the $i$-th row from the top.  Fill the boxes of $\lambda$ with integers $1$ to $n$ starting at the bottom of the leftmost column and moving up the column by increments of one.  Then move to the lowest box of the next column and so on.    This is called the \textbf{base filling} of $\lambda$. Let $X$ be the matrix such that $X_{kj}=1$ if $j$ fills a box directly to the right of $k$ in the base filling and $X_{kj}=0$ otherwise.  
\end{defn}

These matrices will play a key role in the combinatorial results of subsequent sections.

\begin{example}\label{example: highest form} Let $n=5$ and $\lambda = (3,2)$.   Definition \ref{defn: highest form} gives the following base filling of $\lambda$ and nilpotent representative $X$ of Jordan type $\lambda$,
\[
\young(245,13) \quad \mbox{ and }\quad X= \begin{pmatrix} 0 & 0 & 1 & 0 & 0\\ 0 & 0 & 0 & 1 & 0\\ 0 & 0 & 0 & 0 & 0 \\						0 & 0 & 0 & 0 & 1\\ 0 & 0 & 0 & 0 & 0  \end{pmatrix}.
\]
\end{example}

Now we consider the case in which $\fg$ is an arbitrary complex reductive Lie algebra.  In this general setting, it is still possible to choose a representative for a nilpotent $X$ within its conjugacy class so that $X$ is a sum of positive root vectors; moreover, if $X$ is regular in some Levi subalgebra of $\fg$ then it is possible to make this choice so that the Hessenberg Schubert cells form a paving.  The details of this construction are not necessary for our arguments so we refer the interested reader to~\cite[Section 4]{P}.  

Our proofs require the existence of a Hessenberg Schubert paving, which is guaranteed by the following theorem (that combines results of the two authors~\cite{P, T}).

\begin{thm}\label{thm.paving}  Fix a Hessenberg space $H\subseteq \fg$.  Let $X\in \fg$ be a nilpotent element such that $X$ is regular in some Levi subalgebra of $\fg$ and:
\begin{enumerate}
\item if $\fg$ is type A and $X$ has Jordan type $\lambda$, then $X$ is the matrix constructed from the base filling of $\lambda$ as in Definition~\ref{defn: highest form}, or
\item  if $\fg$ is a complex reductive Lie algebra of arbitrary Lie type, then choose $X$ within its conjugacy class as in Section 4 of~\cite{P} \textup{(}c.f.~Corollary 4.9 of~\cite{P}\textup{)}.
\end{enumerate}
Let $X= \sum_{\gamma\in \Phi_X} E_{\gamma}$ for a subset $\Phi_X$ of positive roots. Then the intersection $C_w\cap \B(X,H)$ is nonempty if and only if $wB\in \B(X,H)$ or equivalently $w^{-1}\Phi_X \subseteq \Phi_H$.  If $C_w\cap \B(X,H)$ is nonempty then $C_w\cap \B(X,H)\cong \C^{d_w}$ for some nonnegative integer~$d_w$.  In particular the nonempty Hessenberg Schubert cells pave $\B(X,H)$.
\end{thm}

\begin{rem} If $X\in \mathfrak{gl}_n(\C)$ then $X$ can be conjugated into Jordan form, and Jordan form is regular in the Levi of block-diagonal matrices given by the Jordan blocks.  Results of the first author~\cite{P} and second author in~\cite{T} both prove that a Hessenberg Schubert paving exists in this case.  However, these pavings are obtained by different methods: more precisely, the  representative $X\in \mathfrak{gl}_n(\C)$ used by the first author is not always equal to the matrix from Definition~\ref{defn: highest form}.  We use the latter in this paper, as the matrices associated to the base filling of a Young diagram play a key role in the combinatorial results of subsequent sections.
\end{rem}

\section{Parabolic Hessenberg varieties are pullbacks of Steinberg varieties}\label{parabolic}

In this section we specialize to the case where the Hessenberg space $H$ is a parabolic subalgebra.  After some preliminary discussion, we prove the geometric relationship between parabolic Hessenberg varieties and Steinberg varieties in Theorem~\ref{pullback}. We then use this result to give an explicit formula for the Poincar\'e polynomial of a parabolic Hessenberg variety whenever the Hessenberg Schubert cells form a paving of that variety.

When $G=GL_n(\C)$, a standard parabolic subalgebra consists of all matrices with a particular block upper triangular form.  More generally, a \textbf{parabolic subalgebra} is any Lie subalgebra of $\fg$ containing a Borel subalgebra and similarly for parabolic subgroups.  A classical result states that the subgroups of $G$ containing $B$ are precisely the parabolic subgroups of the form  
\[
P_J = BW_JB = \bigsqcup_{w\in W_J} BwB
\]  
where $J \subseteq \Delta$ is a subset of simple roots and $W_J$ is the subgroup of $W$ generated by $\{ s_{i} \mid \alpha_i \in J \}$ \cite[Theorem 29.3]{H}.  Let $\fp_J=\Lie(P_J)$ denote the corresponding parabolic subalgebra.  Every parabolic subalgebra of this form is a Hessenberg space containing $\fb$.

Denote the projection from the full flag variety $\B=G/B$ to the partial flag variety $G/P_J$ by $\pi_J: G/B \to G/P_J$.  The variety 
\[
\pi_J (\B^X) = \{gP \mid g^{-1}\cdot X\in \fp_J\} \subseteq G/P_J
\]
is called the \textbf{Steinberg variety}.  Steinberg first studied these varieties~\cite{Steinberg}, followed by Shimomura~\cite{Shimomura1, Shimomura2}, and more recently Fresse~\cite{Fresse}. We will recover some of Fresse's results below using a more explicit method that permits us to identify Betti numbers, among other things.

{\bf For the rest of the paper we assume $H=\fp_J$ for some $J\subseteq \Delta$.}  We call the corresponding Hessenberg variety a \textbf{parabolic Hessenberg variety}.  

\subsection{Background on parabolics} We begin with a summary of notation and key structural aspects of parabolics.

Let $\Phi_J \subseteq \Phi$ be the subsystem of roots spanned by $J$ and denote its positive roots by $\Phi_J^+$ and negative roots by $\Phi_J^-$.  The subalgebra $\fp_J$ has Levi decomposition
\[
\fp_J = \fm_J \oplus \fu_J \mbox{ where } \fm_J=  \ft\oplus \bigoplus_{\gamma\in \Phi_J}\fg_{\gamma} 
\mbox{ and } 									
\fu_{J}=\bigoplus_{\gamma\in \Phi^+-\Phi_J^+} \fg_{\gamma}.
\]
There is a corresponding decomposition of $P$ into the semidirect product $M_JU_J$ where $M_J$ and $U_J$ are subgroups of $G$ with  $\Lie(M_J)=\fm_J$ and $\Lie(U_J)=\fu_J$.  Let $M_J/B_J:= M_J/(B\cap M_J)$ denote the flag variety of the Levi subgroup $M_J$.  
 
Each coset in $W/W_J$ contains a unique minimal-length representative.  Denote the set of minimal-length representatives by $W^J$.  This coset decomposition respects lengths; when $w\in W$ is written as $w=vy$ with $v\in W^J$ and $y\in W_J$ then $\ell(w)=\ell(v)+\ell(y)$ \cite[Proposition 2.4.4]{BB}.  The set $W^J$ can be characterized in the following different ways \cite[Remark 5.13]{Ko}.

\begin{rem}\label{fact: shortest coset representatives}   Fix a Weyl group element $v$.  The following statements are equivalent:
\begin{enumerate}
\item The Weyl group element $v$ is in $W^J$.
\item Every positive root $\gamma$ with $v^{-1}(\gamma)\in \Phi^-$ in fact satisfies $v^{-1}(\gamma)\in \Phi^- - \Phi_J^-$. 
\item For all $\alpha_i \in J$, we have $\alpha_i \notin N(v)$.
\end{enumerate}
\end{rem}

The decomposition $W=W^JW_J$ makes the task of identifying inversion sets particularly simple.  This is the context in which we usually use the following lemma, which is also a well-known result \cite[Equation (5.13.2)]{Ko}. 

\begin{lem}\label{fact: inversion sets}  Suppose that $v$ and $y$ are reduced words in $W$ whose product  $w=vy$ is also a reduced word.  Then $\ell(w)=\ell(v)+\ell(y)$ and the inversion set of $w$ is the disjoint union $N(w)=N(y) \sqcup y^{-1}N(v)$.
\end{lem}

The next lemma explicitly describes the projection map $\pi_J: G/B \to G/P_J$.  It is a short reformulation of the previous statements together with classical results that allow us to factor the unipotent subgroup as we wish.  Recall that each Schubert cell $C_w$ can be written as $U^w w B/B$ where $U^w \subseteq U$ is the maximal subgroup such that $w^{-1}U^w w$ is contained in the opposite unipotent, that is $U^w=U\cap wU^-w^{-1}$.  


\begin{lem}\label{proj properties}
Suppose that $w=vy$ with $y\in W_J$ and $v\in W^J$ and that $uwB \in G/B$ is any element of the Schubert cell $C_w$.  Then:
\begin{enumerate}
\item There is a unique way to write $uw$ as $u_1vu_2y$ where $u_1 \in U^v, u_2 \in U^y$.
\item The image of $uwB$ under the map $\pi_J: G/B \rightarrow G/P_J$ is $u_1vP_J$.
\item The preimage of $u_1vP_J$ under the map $\pi$ is $\bigsqcup_{y \in W_J} u_1vU^yyB$.
\item The projection $\pi_J$ restricts to an isomorphism on $C_v$.
\end{enumerate}
\end{lem}

\begin{proof}
Recall that a root subgroup of $U$ is the one-dimensional unipotent subgroup $U_{\gamma}=\exp(\mathfrak{g}_{\gamma})$ for each $\gamma\in \Phi$.  The subgroup $U^w$ is the product $U^w = \prod_{\gamma\in N(w^{-1})} U_{\gamma}$.  Moreover the unipotent subgroup $U$ can be factored as a product of root subgroups in any order \cite[\S 28.1]{H}.  Applying Lemma~\ref{fact: inversion sets} to the factorization $w^{-1}=y^{-1}v^{-1}$ gives $N(w^{-1}) = N(v^{-1}) \sqcup vN(y^{-1})$.  The definition of $U^w$ thus implies $U^w \cong U^v \times vU^yv^{-1}$ proving the first claim.  Since $y\in W_J$ we know $U^y \subseteq U\cap M_J$ and thus $u_2y\in P_J$.  This means $\pi_J(uwB) = u_1vP_J$ proving the second claim.  It now follows that 
\[\pi_J^{-1}(u_1vP_J) \subseteq \bigsqcup_{y \in W_J} u_1vU^yyB.\]
Remark~\ref{fact: shortest coset representatives} states that for each $u_1 \in U^v$ we have $v^{-1}u_1 v \notin P_J$ and so the containment is an equality, proving the third claim.  When restricted to $C_v$ the map $\pi_J$ is surjective (by Claim (2)) and injective (by Claim (3)), completing the proof.
\end{proof} 

\begin{rem}  Claim (4) of the lemma implies that $\pi_J(C_v)$ is the Schubert cell indexed by $v\in W^J$ in $G/P_J$.  We denote this Schubert cell by $C_v^{P_J}$.
\end{rem}


\subsection{The main pullback result}

The next theorem establishes a geometric relationship between the parabolic Hessenberg variety $\B(X,\fp_J)$ and the Springer fiber $\B^X$. It is the main result of this manuscript and holds for all nilpotent $X\in \fb$ and in all Lie types.

\begin{thm}\label{pullback}  Suppose $X\in \fb$ is nilpotent.  The pullback of the Steinberg variety $\pi_J(\B^X)$ under the projection $\pi_J: G/B \to G/P_J$ is the parabolic Hessenberg variety $\B(X,\fp_J)$.
\end{thm}

\begin{proof}  
Since $\B^X \subseteq \B(X,\fp_J)$ we know $\pi_J(\B(X,\fp_J))$ contains the Steinberg variety.  We need only confirm that each $gB \in \B(X,\fp_J)$ is sent to an element $\pi_J(gB) \in \pi_J(\B^X)$ in the Steinberg variety.  Let $gB\in \B(X,\fp_J)$ and write $g=uvp$ for some $u\in U^v$, $v\in W^J$, and $p\in P_J$ per Lemma~\ref{proj properties}.  We will show $uvB\in \B^X$.  Lemma~\ref{proj properties} says $\pi_J(gB) = \pi_J(uvB)$ so this will prove the claim.

By definition of parabolic Hessenberg varieties we know $p^{-1}v^{-1}u^{-1}\cdot X \in \fp_J$.  The parabolic $\fp_J$ is stable under adjoint action of $P_J$ so $v^{-1}u^{-1} \cdot X \in \fp_J$.  Since $X\in \fb$ and $u\in U$, we can write $u^{-1} \cdot X =\sum_{\gamma\in \Phi_Y} c_{\gamma}E_{\gamma}$ for some subset $\Phi_Y$ of positive roots and coefficients $c_{\gamma}\in \C$.  Thus
\[
	v^{-1}\cdot (u^{-1}\cdot X) = \sum_{\gamma\in \Phi_Y} c_{\gamma}E_{v^{-1}(\gamma)}.
\]
If this sum is not in $\mathfrak{b}$ then there is $\gamma\in \Phi_Y$ with $v^{-1}(\gamma)\in \Phi^-$.  We know $v^{-1}\cdot (u^{-1} \cdot X) \in \fp_J$ so $v^{-1}(\gamma)\in \Phi^-_J$.  But Remark \ref{fact: shortest coset representatives} tells us $v^{-1}(\gamma)\in \Phi^- - \Phi_J^-$.  From this contradiction we conclude $v^{-1}(\gamma)\in \Phi^+$ for all $\gamma\in \Phi_Y$ so $v^{-1}u^{-1}\cdot X \in \fb$ and $uvB\in \B^X$ as desired.
\end{proof}

We obtain the following corollary, which gives a formula for the dimension of each Hessenberg Schubert cell in terms of a corresponding Springer Schubert cell (or Steinberg Schubert cell in the partial flag variety $G/P_J$) .

\begin{cor}\label{parabolic dimension}  Fix $J \subseteq \Delta$ and $X\in \fb$. Let $w \in W$ and write $w=vy$ with $v\in W^J$ and $y\in W_J$.  If $wB \in \B(X,\fp_J)$ then 
\[
\dim(C_w\cap \B(X,\fp_J)) = \dim(C_v\cap \B^X) + \ell(y) = \dim(C_v^{P_J} \cap \pi_J(\B^X))+\ell(y).
\]
\end{cor}
\begin{proof} Let $gB\in C_w$ and write $gB=u_1vu_2yB$ for some $u_1\in U^v$ and $u_2\in U^y$ using Lemma~\ref{proj properties}.  Theorem~\ref{pullback} shows
\[
u_1vu_2yB\in C_w\cap \B(X,\fp_J)  \Leftrightarrow u_1vB\in C_v\cap \B(X,\fp_J)  \Leftrightarrow  u_1vB\in C_v\cap \B^X . 
\]
Together with Lemma~\ref{proj properties}, this shows that $\pi_J$ restricts to an isomorphism $C_v\cap \B^X \simeq C_v^{P_J} \cap \pi_J(\B^X)$ and proves the second desired equality. The first equality   also follows from Lemma~\ref{proj properties}, since the map $gB\mapsto (u_2, u_1vB)$ defines an isomorphism of varieties $C_w\cap \B(X,\fp_J) \to U^y \times (C_v \cap \B^X)$. 
\end{proof}

\subsection{Combinatorial corollaries}

We end this section with a collection of combinatorial corollaries of the pullback result.  The key is the following observation that the permutation flags in the parabolic Hessenberg variety $\B(X,\fp_J)$ are precisely the $W_J$-cosets of the permutation flags in the Springer fiber $\B^X$.  

\begin{cor}\label{cor.fixedpts} Let $X\in \fb$ and $w=vy$ with $v\in W^J$ and $y\in W_J$.  Then $wB\in \B(X,\fp_J)$ if and only if $vB\in \B^X$.
\end{cor}

We denote the subset of $W_J$-coset representatives of permutation flags in $\B^X$ by
\[W(X,J):=\{ v\in W^J \mid vB\in \B^X \}\]

\begin{example}\label{(2,2)ex}  Let $X\in \mathfrak{gl}_4(\C)$ be a nilpotent element of Jordan type $\lambda=(2,2)$.  If $X$ is in highest form as in Definition \ref{defn: highest form} then
	\[
		X=\begin{pmatrix} 0 & 0 & 1 & 0\\ 0 & 0 & 0 & 1\\ 0 & 0 & 0 & 0\\ 0 & 0 & 0 & 0 \end{pmatrix}
	\]
and $\Phi_X=\{ \alpha_1+\alpha_2, \alpha_2+\alpha_3 \}$.  If $J=\{ \alpha_1, \alpha_3 \}$ then $W_J$ is the subgroup of $S_n$ generated by $\{s_1, s_3\}$ and $W^J = \{ e, s_2, s_1s_2, s_3s_2, s_1s_3s_2, s_2s_1s_3s_2 \}$. We find the set $W(X,J)$ by checking whether $v^{-1}\cdot X$ is upper triangular for each $v \in W^J$, or equivalently whether $v^{-1}\Phi_X \subseteq \Phi^+_J$.  The following table computes $v^{-1}\alpha$ for each $v \in W_J$ and $\alpha \in \Phi_X$.
\[\begin{array}{c|c|c|c|c|c}
e & s_2 & s_1s_2& s_3s_2 & s_1s_3s_2& s_2s_1s_3s_2  \\

\cline{1-6} \alpha_1+\alpha_2 & \alpha_1&-\alpha_2 & \alpha_1+\alpha_2+\alpha_3&\alpha_3&-\alpha_1-\alpha_2\\
\alpha_2+\alpha_3 & \alpha_3&\alpha_1+\alpha_2+\alpha_3 &-\alpha_2&\alpha_1&-\alpha_2-\alpha_3\\
\end{array}\]

  We conclude $W(X,J)=\{ e, s_2, s_1s_3s_2 \}$.
\end{example}

We can use $W(X,J)$ to describe a paving of the Steinberg variety $\pi_J(\B^X)$ using the projection of the paving by Hessenberg Schubert cells of the parabolic Hessenberg variety $\B(X,\fp_J)$.  When $X$ is in a nilpotent conjugacy class satisfying the assumptions of Theorem~\ref{thm.paving}, this extends and improves on Fresse's result: he proved a paving exists for all Steinberg varieties \cite{Fresse}, but we add explicit information about the cells and their dimensions.  Our results apply to all nilpotents in type $A$, all nilpotents that are regular in a Levi in general type, and some other cases.

\begin{cor}\label{cor.paving} Suppose $X\in \fb$ is a nilpotent element satisfying the assumptions of Theorem~\ref{thm.paving}.  Then the intersection $C_v^{P_J} \cap \pi_J(\B^X)$ is nonempty if and only if $v\in W(X,J)$.  Furthermore, if $v\in W(X,J)$ then $C_v^{P_J} \cap \pi_J(\B^X)\simeq \C^{d_v}$ where $d_v = \dim(C_v\cap \B^X)$. 
\end{cor}
\begin{proof} Let $v\in W^J$. 
By Theorem~\ref{thm.paving} the cell $C_v\cap \B^X$ is nonempty if and only if $vB\in \B^X$.  The condition $vB \in \B^X$ is equivalent to $v \in W(X,J)$ by definition and to $vP \in \pi_J(\B^X)$ by Lemma~\ref{proj properties}.  The map $\pi_J$ restricts to an isomorphism $C_v\cap \B^X\simeq C_v^{P_J}\cap \pi_J(\B^X)$ so $C_v^{P_J}\cap \pi_J(\B^X)$ is nonempty if and only if $v\in W(X,J)$ in which case it has the same dimension as $C_v \cap \B^X$.  Finally, if $v\in W(X,J)$ then $C_v\cap \B^X \simeq \C^{d_v}$ by Theorem~\ref{thm.paving}.
\end{proof}

\begin{rem}
A priori, Corollary~\ref{cor.paving} only applies to those $X\in \mathfrak{gl}_n(\C)$ corresponding to the base filling of the partition $\lambda$ obtained by recording the sizes of the Jordan blocks of $X$ \textup{(}see~Definition~\ref{defn: highest form}\textup{)}.  However each $X'\in \mathfrak{gl}_n(\C)$ is conjugate to an $X'$ of the desired form.  Conjugating $X'$ is equivalent to translating the Springer fiber, in the sense that $\B^{g^{-1} \cdot X'} = g^{-1}\B^X$.  Since pavings are preserved under translation, we conclude that {\em all} Steinberg varieties $\pi_J(\B^{X'})$ are paved by affines in type $A$.
\end{rem}

Using these results, we prove the second main theorem of this section: a factorization of the Poincar\'e polynomial of a parabolic Hessenberg variety into the product of the Poincar\'{e} polynomials of a Steinberg variety and the flag variety of the Levi subgroup $M_J$.  We denote the Poincar\'{e} polynomial in variable $t$ of a variety $\mathcal{X}$ by $\mathsf{P}(\mathcal{X}, t)$.  Recall that $M_J/B_J = M_J/(B\cap M_J)$ denotes the flag variety of the Levi subgroup $M_J$. Note that the permutation flags of $M_J/B_J$ are precisely $y(B \cap M_J)$ for $y \in W_J$.

\begin{thm}\label{Poincare polynomial} Suppose $X\in \fb$ is a nilpotent element satisfying the assumptions of Theorem~\ref{thm.paving}.  Let $J\subseteq \Delta$. Then
\[
\mathsf{P}(\B(X,\fp_J),t) =  \mathsf{P}(\pi_J(\B^X), t)\mathsf{P}(\B_J, t).
\]
\end{thm}
\begin{proof}  By Corollary~\ref{cor.paving}, the intersections $C_v\cap \pi_J(\B^X)$ with $v\in W(X,J)$ pave $\pi_J(\B^X)$ and thus give the Betti numbers of the Steinberg variety (see Remark~\ref{betti}).  Since $\pi_J$ restricts to an isomorphism on $C_v \cap \B^X$ we  write
\begin{eqnarray}\label{eqn.polynomial}
\mathsf{P}(\pi_J(\B^X), t) = \sum_{v\in W(X,J)} t^{\dim (C_v\cap \pi_J(\B^X))} = \sum_{v\in W(X,J)} t^{\dim (C_v\cap \B^X)}.
\end{eqnarray}
Theorem~\ref{thm.paving} says that the nonempty intersections $C_w\cap \B(X,\fp_J)$ pave the Hessenberg variety $\B(X, \fp_J)$.  Corollary~\ref{cor.fixedpts} says $C_w\cap \B(X,\fp_J)\neq \emptyset$ if and only if $w=vy$ with $y\in W_J$ and $v\in W(X,J)$.  Applying Corollary~\ref{parabolic dimension}, we obtain:
\begin{eqnarray*}
\mathsf{P}(\B(X,\fp_J),t) &=& \sum_{v\in W(X,J)} \sum_{y\in W_J} t^{\dim(C_v\cap \B^X) }t^{\ell(y)}\\
	&=& \sum_{v\in W(X,J)} t^{\dim(C_v\cap \B^X)}\;\;  \sum_{y\in W_J} t^{\ell(y)} \\ 
	&=& \mathsf{P}(\pi_J(\B^X), t)\mathsf{P}(M_J/B_J, t)
\end{eqnarray*} 
which proves the desired result. \end{proof}



The next section strengthens these combinatorial results in the case of type $A$. Example~\ref{example: factorization formula} below demonstrates how Theorems~\ref{pullback} and~\ref{Poincare polynomial} can be used in that setting.


\section{Application in type $A$:  Betti numbers of Steinberg varieties}\label{section: Betti numbers of Steinberg} 

We give two main applications in type $A$.  The first, given in this section, computes the Betti numbers of Steinberg varieties using the combinatorics of row-semistrict tableaux.  The second, given in the next section, will show that the Betti numbers of parabolic Hessenberg varieties and Steinberg varieties match those of specific unions of Schubert varieties whenever the Jordan form of $X$ corresponds to a partition with at most three row or two columns.  

We begin with a subsection that summarizes the key combinatorial objects in the case of type $A$, especially tableaux and the kinds of inversions within tableaux that count dimensions in pavings of Springer fibers.  The second subsection adapts these combinatorial descriptions to partial flag varieties, combining them with the results in Section~\ref{parabolic} to give an explicit description of the Betti numbers of Steinberg varieties.

\subsection{Notation for type $A$}
When $\fg=\mathfrak{gl}_n(\C)$ both $X$ and $P_J$ are determined by partitions.  Let $\mu= (\mu_1,\mu_2, \ldots, \mu_k)$ be a partition of $n$. Associate a subset of simple roots to $\mu$ by the rule that
\[
J_\mu = \Delta \setminus \{ \alpha_{\mu_1}, \alpha_{\mu_1+\mu_2}, \ldots, \alpha_{\mu_1+\cdots+\mu_{k-1}} \}.
\]
The corresponding parabolic subalgebra $\fp_J$ for $J=J_\mu$ is the subalgebra of block-upper-triangular matrices whose block-sizes are determined by $J$.  Every subset $J\subseteq \Delta$ has the form $J=J_{\mu}$ for some composition $\mu$.  However we gain no generality by using compositions for $\mu$ since reordering blocks corresponds to conjugating the parabolic, which in turn induces an isomorphism $G/P \simeq G/(wPw^{-1})$.  

Let $\lambda$ be a partition of $n$. We let $X$ be the highest form representative of the conjugacy class of nilpotent matrices of Jordan type $\lambda$, as given in Definition~\ref{defn: highest form}.  

The permutation flags $wB$ in the Springer fiber $\B^X$ are in bijection with the row-strict tableaux, namely tableaux whose entries increase from left to right in each row.  The following result describes this bijection explicitly \cite[Theorem 7.1]{T}.

\begin{lem}\label{lem.points}  The permutation flag $wB$ is an element of $\B^X$ if and only if the tableau $T$ of shape $\lambda$ given by labeling the $i$-th box in the base filling of Definition~\ref{defn: highest form} by $w^{-1}(i)$ is a row-strict tableau.
\end{lem} 

For example, the identity permutation corresponds to the base filling of $\lambda$.  More generally, note that if $i$ labels a box in $T$ then the corresponding box in the base filling of $\lambda$ is labeled by $w(i)$.  

Not only do the row-strict tableaux of shape $\lambda$ index the nonempty Springer Schubert cells $C_w \cap \B^X$ but they encode the dimensions $\dim(C_w\cap \B^X)$.  The next lemma explains how, by counting certain inversions in the tableau $T$.  (It is an amalgamation of several earlier results that are itemized in the proof.) 

Let $\RST(\lambda)$ denote the set of all row-strict tableaux of shape $\lambda$.   Let $T$ be a row-strict tableau and $T[i]$ be the diagram obtained by restricting $T$ to the boxes labeled $1,\ldots,i$.  (Since $T$ is row-strict, the diagram $T[i]$ consists of rows of boxes without gaps in rows---in other words if a box is deleted, all boxes in the same row and to the right of that box must also have been deleted.)

\begin{lem}\label{counting rows} Suppose $wB\in \B^X$ and let $T\in \RST(\lambda)$ be the row-strict tableau corresponding to $w$ as in Lemma~\ref{lem.points}. Let $2\leq q \leq n$ and $\ell_{q-1}$ be the sum of
\begin{itemize}
\item the number of rows in $T[q]$ above the row containing $q$ and of the same length, plus 
\item the total number of rows in $T[q]$ of strictly greater length than the row containing $q$.
\end{itemize}
Then 
\[
\dim(C_w\cap \B^X) = \sum_{i=2}^n \ell_{i-1}
\]
We call $\ell_{q-1}$ the \textbf{number of $q$-row inversions of the diagram $T$}.  
\end{lem}

\begin{proof}
Springer dimension pairs are a subset of the inversions in a filled tableau; the total number of Springer dimension pairs is equal to $\dim(C_w\cap \B^X)$ by work of the second author \cite[Theorem 7.1]{T}.   A Springer dimension pair $(p,q)$ satisfies:
\begin{enumerate} 
\item $1 \leq p<q \leq n$ and 
\item $q$ occurs in a box below $p$ and in the same column or in any column strictly to the left of $p$ in $T_v$ and
\item if the box directly to the right of $p$ in $T_v$ is filled by $r_{p}$ then $q\leq r_{p}$.
\end{enumerate}
The quantities $\ell_{q-1}$ count the number of Springer dimension pairs  of the form $(p,q)$ for $1 \leq p < q \leq n$ and so the sum of the $\ell_{q-1}$ also gives the total number of Springer dimension pairs \cite{PT, M}. 
\end{proof}

\begin{example}\label{ex.(2,2)contd} Continuing~Example \ref{(2,2)ex}, let $\lambda=(2,2)$ and $X\in \mathfrak{gl}_4(\C)$ be the corresponding nilpotent matrix. The following table displays all row-strict tableaux of shape $(2,2)$, records the corresponding permutation $w\in S_4$ such that $wB\in \B^X$, and computes $\dim(C_w\cap \B^X)$. 
\begin{center}
\begin{tabular}{l | c | c | c | c | c | c}
 \empty & \multirow{3}{*}{$\young(24,13)$} & \multirow{3}{*}{$\young(34,12)$} & \multirow{3}{*}{$\young(14,23)$} & \multirow{3}{*}{$\young(23,14)$} & \multirow{3}{*}{$\young(13,24)$} & \multirow{3}{*}{$\young(12,34)$}\\
 $T$ & & & & &\\
 & & & & &\\\hline
$w\in S_4$ & $e$ & $s_2$ & $s_1$ & $s_3$ & $s_1s_3$ & $s_1s_3s_2$ \\ \hline
$\dim(C_w\cap \B^X)$ & $0$ & $1$ & $1$ & $1$ & $2$ & $2$\\
\end{tabular}
\end{center}
For example, to see $\dim(C_{s_1s_3s_2}\cap \B^X)=2$ we compute $\ell_3=1$ \textup{(}since $T=T[4]$ has one row of length $\geq 2$ other than the row containing $4$\textup{)}, $\ell_2=1$ \textup{(}since $T[3]$ has one row of length $\geq 1$ other than the row containing $3$\textup{)}, and $\ell_1=0$ \textup{(}since $T[2]$ has only one row\textup{)}. 
\end{example}

\begin{example}\label{example: factorization formula}
We use Example~\ref{ex.(2,2)contd} to give an explicit example of the results from Section~\ref{parabolic}.  As in Example~\ref{(2,2)ex}, take $J=J_{(2,2)}=\{\alpha_1, \alpha_3\}$ so $W(X,J)=\{e, s_2, s_1s_3s_2\}$.  The Poincar\'e polynomial of the Steinberg variety $\pi_J(\B^X)$ is determined by the dimensions $\dim(C_v\cap \B^X)$ above when $v\in W(X,J)$.  Thus we have $\mathsf{P}(\pi_J(\B^X), t) = 1+t+t^2$.  

Since $W_{J}=\{ e, s_1, s_3, s_1s_3 \}$ Theorem~\ref{Poincare polynomial} gives the Poincar\'e polynomial of $\B(X,\fp_{(2,2)})$:
\begin{eqnarray*}
\mathsf{P}(\B(X,\fp_J), t) =  (1+t+t^2)(1+2t+t^2)= 1+3t+4t^2+3t^3+t^4.
\end{eqnarray*}
\end{example}

\subsection{Betti numbers of Steinberg varieties}  Using the main theorems of Section~\ref{parabolic}, we prove that the Betti numbers of Steinberg varieties are enumerated by row-semistrict  tableaux.  

\begin{defn}\label{def.semistandard} Let $\lambda$ and $\mu$ be partitions of $n$.  A \textbf{row-semistrict tableau of shape $\lambda$ and weight $\mu$} is a tableau $T$ of shape $\lambda$ with $\mu_1$ many $1$'s, $\mu_2$ many $2$'s, and so on, such that the entries in each row are weakly increasing. Let $\RSST(\lambda, \mu)$ denote the set of all row-semistrict tableaux of $\lambda$ and weight $\mu$.  If the entries in each column of $T$ are strictly increasing, then we say that $T$ is a \textbf{semistandard tableau of shape $\lambda$ and weight $\mu$} and let $\SST(\lambda,\mu)$ denote the subset of $\RSST(\lambda, \mu)$ of semistandard tableaux.
\end{defn}

There is a natural map from row-strict tableaux of shape $\lambda$ to row-semistrict tableaux of shape $\lambda$ and content $\mu$ obtained simply by repeating entries.  More precisely, relabel the first $\mu_1$ integers $1$, the next $\mu_2$ integers $2$, the next $\mu_3$ integers $3$, and so on.  For example, if $\mu=(3,2)$ then $1,2,3\mapsto 1$ and $4,5\mapsto 2$.  The \textbf{degeneration map} $\phi_{\lambda,\mu}: \RST(\lambda) \to \RSST(\lambda, \mu)$ is induced on row-strict tableaux by this relabeling.  

\begin{example}\label{example: RSST map} If $\lambda=\mu=(2,2)$ then $1,2 \mapsto 1$ and $3,4 \mapsto 2$ and thus:
\[
\phi_{(2,2),(2,2)}\left( \hspace{0.25em} \young(13,24) \hspace{0.25em} \right) = \young(12,12) \quad \textup{ and } \quad 
\phi_{(2,2),(2,2)}\left( \hspace{0.25em} \young(12,34) \hspace{0.25em} \right) = \young(11,22)\,.
\]
\end{example}

The degeneration map is not typically injective.  However, the next lemma tells us that when restricted   to the row-strict tableaux corresponding to $W(X,J_\mu)$, the degeneration map is bijective.  Let $\RST(\lambda,\mu)$ denote the set of all row-strict tableaux of shape $\lambda$ corresponding to $v\in W(X,J_\mu)$, namely obtained by labeling the $i$-th box in the base filling of $\lambda$ by $v^{-1}(i)$ for each $i$.  We have the following four related objects, which we collect here for the reader's convenience:
\begin{itemize}
\item $\RST(\lambda)$ is the set of all row-strict tableaux of shape $\lambda$
\item $\RST(\lambda,\mu)$ is the set of all row-strict tableaux of shape $\lambda$ corresponding to $v\in W(X,J_\mu)$
\item $\RSST(\lambda, \mu)$ is the set of all row-semistrict tableaux of shape $\lambda$ and weight $\mu$
\item $\SST(\lambda,\mu)$ is the set of semistandard tableaux of shape $\lambda$ and weight $\mu$.
\end{itemize}
The next result shows that $\phi_{\lambda \mu}$ is bijective on $\RST(\lambda,\mu)$ while a later result studies the preimage under $\phi_{\lambda \mu}$ of $\SST(\lambda,\mu)$.

\begin{lem}\label{lem.semistandard}   The restriction of the degeneration map to $\RST(\lambda,\mu)$ is bijective:
\[
\phi_{\lambda, \mu}: \RST(\lambda,\mu) \xrightarrow{\;\;\sim\;\;} \RSST(\lambda, \mu)
\]
\end{lem}
\begin{proof} We define a map $\psi_{\lambda,\mu}: \RSST(\lambda, \mu) \to \RST(\lambda, \mu)$ and prove that it is the inverse of $\phi_{\lambda, \mu}$.  

Let $T\in \RSST(\lambda, \mu)$.  The boxes of $T$ that are labeled by a fixed $i\in [k]$ are totally ordered by the base filling of $\lambda$.  Label these boxes, in order, with the integers $\mu_0+\mu_1+\cdots+\mu_{i-1}+1, \ldots , \mu_1 + \cdots+\mu_i$.  Proceeding in this fashion for each $i\in [k]$ gives a row-strict tableau, denoted $\psi_{\lambda, \mu}(T)\in \RST(\lambda,\mu)$.  By construction $\phi_{\lambda, \mu}\circ \psi_{\lambda,\mu}(T)=T$ for all $T\in \RSST(\lambda, \mu)$.

To complete the proof, we show $\psi_{\lambda, \mu}(T)$ corresponds to $v \in W(X,J_{\mu})$ (in the sense of Lemma~\ref{lem.points}) for each $T\in \RSST(\lambda, \mu)$.  By construction, writing the numbers that fill $\psi_{\lambda, \mu}(T)$ in order of the base filling of $\lambda$ gives the sequence $[v^{-1}(1), v^{-1}(2), \cdots , v^{-1}(n)]$ that is the one-line notation for $v^{-1}$.  Also by construction, the first $\mu_1$ numbers in this sequence are in increasing order, as are the next $\mu_2$, the $\mu_3$ after that, and so on.  Thus given a pair $p<q$ with $v^{-1}(p)>v^{-1}(q)$ we know that $p,q$ are in different ``blocks", meaning they cannot be a pair of the following form:
\[
\{(i,j)\mid  \mu_0+\cdots+\mu_{i-1}+1< p,q \leq \mu_1+\cdots+\mu_i \textup{ for some } i\in [k]  \},
\]
But the pairs $(p,q)$ in those ``blocks" are precisely the indices corresponding to the roots $\Phi_J$.  We have confirmed the condition in statement (2) of Remark~\ref{fact: shortest coset representatives} holds for $v$ so $v \in W^{J_\mu}$ and hence $v\in W(X,J_\mu)$.  Thus $\psi_{\lambda, \mu}\circ \phi_{\lambda, \mu}$ restricts to the identity on $\RST(\lambda, \mu)$, as desired. 
\end{proof}

\begin{example}\label{example: finding v_T}
Continuing the previous example, we observe that $\psi_{\lambda, \mu}$ sends
\[\young(12,12) \mapsto \young(24,13) \quad \textup{ and } \quad \young(11,22) \mapsto \young(12,34)\]
In both cases we have $\phi_{\lambda,\mu}(\psi_{\lambda,\mu}(T))=T$.  
\end{example}

The following proposition is a version of Lemmas~\ref{lem.points} and~\ref{counting rows} for Steinberg varieties.  Although similar descriptions of the irreducible components of Steinberg varieties have appeared in the literature \cite{Shimomura1, Shimomura2, Steinberg}, the formula below computes the entire Poincar\'e polynomial.  There are similar formulas for the Betti numbers of a different generalization of Springer fibers to partial flag varieties called \textit{Spaltenstein varieties} \cite{MR3853893, MR2827037}.  

\begin{prop}\label{prop.semistandard} Let $\lambda$ and $\mu$ be partitions of $n$ and assume $\mu$ has $k$ rows.  Let $X$ be the matrix in the nilpotent conjugacy class associated to $\lambda$ given in Definition~\ref{defn: highest form} and $J=J_\mu$.  For each $T\in \RSST(\lambda, \mu)$ let $d_T$ be the number of pairs $(p,q)\in [k]\times [k]$ {\bf{counted with multiplicity}} such that 
\begin{enumerate}
\item $p<q$ and
\item $q$ occurs in a box below $p$ and in the same column or in any column strictly to the left of $p$ in $T$ and
\item if the box directly to the right of $p$ in $T$ is filled by $r_p$ then $q\leq r_p$
\end{enumerate}
 Then
 \[
\mathsf{P}(\pi_J(\B^X), t) = \sum_{T\in \RSST(\lambda, \mu)} t^{d_T}.
\]
\end{prop}

\begin{proof} By Corollary~\ref{cor.paving}, the intersections $\{C_v\cap \pi_J(\B^X)\mid v\in W(X,J)\}$ pave $\pi_J(\B^X)$ and moreover $\dim(C_v\cap \pi_J(\B^X)) = \dim(C_v\cap \B^X)$. Lemma~\ref{lem.semistandard} shows that each $T \in \RSST(\lambda,\mu)$ corresponds to a unique $v \in W(X,J)$ since  $\phi_{\lambda,\mu}^{-1}(T) \in \RST(\lambda,\mu)$. Thus it suffices to show that $\dim(C_v\cap \B^X) = d_T$ for each $T \in \RSST(\lambda,\mu)$ whenever $v \in W(X,J_\mu)$ is the permutation corresponding to the tableau $T_v = \phi_{\lambda,\mu}^{-1}(T)$.  

By definition $\phi_{\lambda,\mu}(T_v) = T$.  The conditions on $(p,q)$ in Proposition~\ref{prop.semistandard} are precisely those from the proof of Lemma~\ref{counting rows} counting inversions in $T_v$.  Thus $\dim(C_v \cap \B^X) \geq d_T$ for each $v \in W(X,J_{\mu})$.  By Proposition~\ref{prop.semistandard} if $p'<q'$ satisfy $\mu_0+\mu_1+\cdots+\mu_i< p',q'\leq \mu_1+\cdots +\mu_{i}$ for some $i\in [k]$ then $v^{-1}(p')<v^{-1}(q')$.  Thus the degeneration map sends each inversion $(p',q')$ in $T_v$ to a pair $(p,q) \in [k] \times [k]$ with $p \neq q$ and so $(p,q)$ contributes to $d_T$.  This means $\dim(C_v \cap \B^X) = d_T$ and the claim is proved.
\end{proof}

\begin{example}\label{ex.semistandard} Let $\lambda=\mu=(2,2)$ as in Example~\ref{ex.(2,2)contd}.  The table below displays the three row-semistrict tableaux in $\RSST(\lambda,\mu)$ and the pairs counted by $d_T$ in each case.
\[
\begin{array}{c|c|c|c}
T\in \RSST(\lambda,\mu) & \young(12,12) & \young(22,11) & \young(11,22) \\
&&&\\
\cline{1-4} &&& \\
\textup{  pairs counted by $d_T$ } & \emptyset & (1,2) & (1,2), (1,2) \end{array}
\]
The pair $(1,2)$ is counted twice for the last row-semistrict tableau since there are two pairs satisfying the given conditions---one for each $2$ appearing in the second row of~$T$. 
\end{example}

By Corollary~\ref{cor.paving}, the dimension of the Steinberg variety $\pi_J(\B^X)$ is 
\[
\max \{ \dim(C_v\cap \B^X) \mid v\in W(X,J) \} \leq \dim(\B^X).
\]
Steinberg first counted the irreducible components of $\pi_J(\B^X)$ with maximal dimension $\dim(\B^X)$ in~\cite{Steinberg}.  The following corollary is a simpler proof of Steinberg's theorem, using only the affine paving and combinatorics of row-strict tableaux.  Recall that the {\em Kostka number} $K_{\lambda\mu}$ is the number of semistandard tableaux of shape $\lambda$ and weight $\mu$.  The Kostka number is an important quantity in algebraic combinatorics and representation theory.

\begin{cor}\label{cor.irredcomp} Let $\lambda$ and $\mu$ be partitions of $n$, $X\in \mathfrak{gl}_n(\C)$ the nilpotent matrix of Jordan type $\lambda$ fixed in Definition~\ref{defn: highest form}, and $J=J_\mu$.   There are exactly $K_{\lambda\mu}$ irreducible components of $\pi_J(\B^X)$ of dimension $\dim(\B^X)$.
\end{cor}

\begin{proof} 
First we identify the irreducible components of $\pi_J(\B^X)$ of dimension $\dim(\B^X)$.  Corollary~\ref{cor.paving} showed that $C_v^{P_J} \cap \pi_J(\B^X)$ is isomorphic to affine space so $\overline{C_v^{P_J}\cap \pi_J(\B^X)}$ is irreducible and nonempty for all $v\in W(X,J)$.  Furthermore if $\dim(C_v^{P_J}\cap \pi_J(\B^X)) = \dim(\B^X)$ then $\overline{C_v^{P_J}\cap \pi_J(\B^X)}$ must be an irreducible component.  If $v \in W(X,J)$ then Corollary~\ref{cor.paving} said $\dim(C_v^{P_J}\cap \pi_J(\B^X)) = \dim(C_v \cap\B^X)$. Finally, the dimension of $C_v\cap \B^X$ is maximal if and only if the corresponding row-strict tableau $T\in \RST(\lambda)$ is in fact a standard tableau (e.g.~\cite[Theorem 3.5]{PT}).  Thus we need to find the set of $v \in W(X,J)$ that correspond to standard tableaux.

To complete the proof, we argue that there are $K_{\lambda\mu}$ many such $v$.    We know that $\phi_{\lambda,\mu}: \RST(\lambda,\mu) \rightarrow \RSST(\lambda,\mu)$ is a bijection by Lemma~\ref{lem.semistandard}.  If $T \in \RSST(\lambda,\mu)$ is {\em not} semistandard--namely there is a column in which some $i$ appears twice--then its row-strict preimage is not column-strict, since the base filling of $\lambda$ increases bottom-to-top in columns.  If $T$ is semistandard then its row-strict preimage is column-strict by construction of the inverse map, and hence is standard.  Thus the unique preimage in $\RST(\lambda,\mu)$ of each  semistandard $T$ of shape $\lambda$ and weight $\mu$ must be standard.  The tableaux in $\RST(\lambda,\mu)$ are precisely those corresponding to $W(X,J)$ so this proves the claim.
\end{proof}

\begin{example}
Example~\ref{ex.semistandard} showed that when $\lambda=\mu=(2,2)$ the Steinberg variety $\B(X,\fp_J)$ has a single irreducible component of dimension $\dim(\B^X)=2$.  A key property of Kostka numbers is that $K_{\lambda\lambda}=1$ for all $\lambda$.  This confirms the results of Corollary~\ref{cor.irredcomp} in this case.  
\end{example}

We can use other classical properties of Kostka numbers to infer data about Steinberg varieties.  For instance, recall that $K_{\lambda\mu}=0$ whenever $\mu \not\trianglelefteq \lambda$, where $\trianglelefteq$ denotes the dominance order on partitions of $n$.  Corollary~\ref{cor.irredcomp} implies that the dimension of the Steinberg variety $\pi_J(\B^X))$ is strictly less than that of the Springer fiber $\B^X$ whenever $J=J_\mu$, $X$ is of Jordan type $\lambda$, and $\mu \not\trianglelefteq \lambda$.  In Section~\ref{Conclusion} we give an explicit example in which this occurs.


\section{Applications in type $A$: Parabolic Hessenberg varieties have the same Poincar\'{e} polynomial as unions of Schubert varieties}\label{Schubert points}

Our second application of the main theorem identifies specific unions of Schubert varieties whose Poincar\'{e} polynomials agree with those of parabolic Hessenberg varieties.  We use the same notation as in the previous section, again just treating type $A$.  Our strategy is to associate to each flag $wB \in \B(X,\fp_J)$ a permutation $w_T$ whose length is the dimension $\dim(C_w\cap \B(X,\fp_J))$ of the Hessenberg Schubert cell for $wB$.  We call $w_T$ the {\em Schubert point} corresponding to $w$.  We will show that the map $w \mapsto w_T$ preserves the set $W^J$.  We use this together with the decomposition $w_T = v_T y$ into a product of $v_T \in W^J$ and $y \in W_J$ to construct Schubert varieties whose permutation flags are a union of $W_J$-cosets.
Theorem~\ref{Main Theorem} proves that if $X\in \mathfrak{gl}_n(\C)$ is a matrix whose Jordan form corresponds to a partition with at most three rows or two columns, the Betti numbers of  $\B(X,\fp_J)$ match those of 
\[
\bigcup_{v\in W(X,J)} \overline{C}_{v_Tw_J}
\] 
where $w_J\in W_J$ denotes the longest element of $W_J$.   The theorem also gives an analogue for $\pi_J(\B^X)$.

Any parabolic Hessenberg variety that is not irreducible will correspond to the union of more than one Schubert variety.  The Schubert cells in their intersection are counted only once, not with multiplicity, which is the main subtlety of this theorem.

We begin with a canonical factorization of $W=S_n$ following Bj\"{o}rner-Brenti's presentation \cite[Corollary 2.4.6]{BB}.  Recall that the roots associated to the $i^{th}$ row of an upper-triangular matrix are
\[
\Phi_i=\{ \alpha_i, \alpha_i+\alpha_{i+1}, ..., \alpha_i+ \alpha_{i+1}+\cdots + \alpha_{n-1} \} \textup{ for each } 1\leq i \leq n-1.
\] 

\begin{lem}[Bj\"{o}rner-Brenti] \label{fact: strings} Each $w\in W$ can be written uniquely as $w=w_{n-1}w_{n-2}\cdots w_2w_1$ where 
	\[
		w_i=s_{k_i} s_{k_i+1} \cdots s_{i-1} s_i \textup{   for each  } i=1,...,n-1
	\] 
and either $w_i=e$ or $k_i$ is a fixed integer with $1\leq k_i \leq i$.   We call $w_i$ the $i$-th string of $w$.  Moreover
	\[
		w_1^{-1}w_2^{-1}\cdots w_{i-1}^{-1} N(w_i)\subseteq \Phi_i  \textup{   for each  } i=1,...,n-1.
	\]
\end{lem}

For example the longest word in $S_4$ can be written as $s_1s_2s_3s_1s_2s_1$.  In this case the strings are
\begin{itemize}
\item $w_3=s_1s_2s_3$
\item $w_2=s_1s_2$ and
\item $w_1=s_1$
\end{itemize} 
so $k_i=1$ for all $i=1,2,3$.   Note that if $w_i \neq e$ then $\ell(w_{i})=i-k_i+1$.  

In previous work the authors studied a bijection between $wB \in \B^X$ and certain permutations $w_{T}\in W$ whose lengths are the dimension of the corresponding Springer Schubert cells \cite[Definition 3.2]{PT}.  We define those permutations now.

\begin{defn}\label{defn: Schubert point}  Let $wB\in \B^X$ and let $T$ denote the corresponding row-strict tableau as in Lemma~\ref{lem.points}.  For each $2\leq q \leq n$ let $\ell_{q-1}$ be the number of $q$-row inversions of $T$ given in Lemma~\ref{counting rows}.  Define a string $w_{q-1}$ by 
\[
	w_{q-1} = \left\{  \begin{tabular}{l l} $s_{q-\ell_{q-1}} s_{q-\ell_{q-1}+1}\cdots s_{q-2}s_{q-1}$ & if $\ell_{q-1} \neq 0$\\
								$e$ & if $\ell_{q-1} =0$ \end{tabular} \right.
\]
so $w_{q-1}$ is a string of length $\ell_{q-1}$ by construction.  Then
	\[
		w_{T}=w_{n-1} w_{n-2} \cdots w_2 w_1
	\]
is the \textbf{Schubert point} associated to $wB\in \B^X$.  
\end{defn}

By construction 
\[\ell(w_{T})=\ell_{n-1}+ \ell_{n-2}+\cdots + \ell_1=\dim(C_w\cap \B^X).\]
In fact not only are the permutations $w_T$ in bijection with row-strict tableaux, but the set of Schubert points $\{w_T\mid  \mbox{$T$ is row-strict}\}$ forms a lower order ideal in the Bruhat graph whenever $\lambda$ has at most three rows or two columns---namely the Schubert points index a union of Schubert varieties \cite[Theorem 4.4]{PT}.

\begin{lem}[Precup-Tymoczko] \label{schubertpoint} For each $wB\in \B^X$ there exists a unique Schubert point $w_T\in W$.  In addition, if the Jordan form of $X$ corresponds to a partition with at most three rows or two columns then every permutation $w' \leq w_T$ in Bruhat order corresponds to a unique $yB\in \B^X$ such that $w'=y_{T'}$ for the row-strict tableau $T'$ corresponding to $y$.  
\end{lem}

Our plan to extend this result is to show that the Schubert points respect the decomposition $W^JW_J$.  More precisely we will show that $v\in W^J$ if and only if the Schubert point $v_{T}$ corresponding to $v$ is an element of $W^J$.  We begin with an alternate characterization of $W^J$. 

\begin{prop}\label{shortest coset characterization}  Let $w\in W$ and write $w= w_{n-1}w_{n-2}\cdots w_2 w_1$ where $w_i$ denotes the $i$-th string of $w$ for each $i=1,2,\ldots,n-1$.  Then $w\in W^J$ if and only if $\ell(w_i)\leq \ell(w_{i-1})$ for all $\alpha_i \in J$.
\end{prop}
\begin{proof}  We will prove the contrapositive statement using Remark \ref{fact: shortest coset representatives}, which says that $w$ is {\em not} in $W^J$ if and only if there is a simple root $\alpha_i \in J$ for which $\alpha_i \in N(w)$.  In particular we prove that for each simple root $\alpha_i \in J$, the root $\alpha_i \in N(w)$ if and only if $\ell(w_i) > \ell(w_{i-1})$.  

Since $\ell(w)=\ell(w_{n-1})+\ell(w_{n-2})+\cdots + \ell(w_2)+\ell(w_1)$ we can write
\[
N(w) = N(w_1) \sqcup w_1^{-1}N(w_2) \sqcup \cdots \sqcup w_1^{-1}w_2^{-1}\cdots w_{n-2}^{-1}N(w_{n-1})
\]
by Lemma \ref{fact: inversion sets}.  Given $\alpha_i \in J$ consider $w_i=s_{k_i}s_{k_i+1}\cdots s_{i-1}s_i$ and $w_{i-1}=s_{k_{i-1}} s_{k_{i-1}+1}\cdots s_{i-2}s_{i-1}$.  Note that
\begin{eqnarray}\label{inversion set}
N(w_i) = \{ \alpha_i, s_i(\alpha_{i-1}), ..., s_is_{i-1}\cdots s_{k_i+1}(\alpha_{k_i}) \}.
\end{eqnarray}
By Lemma \ref{fact: strings} we know $\alpha_i\in N(w)$ if and only if $\alpha_i \in w_1^{-1}w_2^{-1} \cdots w_{i-2}^{-1}w_{i-1}^{-1} N(w_i)$.  Since $\ell(w_i)= i-k_i +1$ we know
	\begin{eqnarray*}
		\ell(w_i)> \ell(w_{i-1}) \hspace{0.1in} \Leftrightarrow \hspace{0.1in} i-k_{i}+1> i-1 -k_{i-1}+1.
	\end{eqnarray*}
This in turn is equivalent to $k_i  \leq k_{i-1}$ and implies that the reflection $s_{k_{i-1}}$ must occur in the word $w_i=s_{k_i}s_{k_i+1}\cdots s_{i-1}s_i$.  The description of $N(w_i)$ in Equation (\ref{inversion set}) shows that this is the case if and only if 
	\[
		s_i s_{i-1}\cdots s_{k_{i-1}+1}(\alpha_{k_{i-1}}) = \alpha_{k_{i-1}}+\alpha_{k_{i-1}+1} + \cdots + \alpha_{i-1}+\alpha_i \in N(w_i).
	\]
Thus $k_i \leq k_{i-1}$ if and only if 
\begin{eqnarray*}
w_1^{-1}w_2^{-1} \cdots w_{i-2}^{-1}w_{i-1}^{-1} (\alpha_{k_{i-1}}+\alpha_{k_{i-1}+1} + \cdots + \alpha_{i-1}+\alpha_i ) \in N(w)\end{eqnarray*}
But 
\begin{eqnarray*}
&&w_{i-1}^{-1}(\alpha_{k_{i-1}}+\alpha_{k_{i-1}+1} + \cdots +\alpha_{i-1}+ \alpha_i )=\\
&&\;\;\;s_{i-1}s_{i-2}\cdots s_{k_{i-1}+1}s_{k_{i-1}}(\alpha_{k_{i-1}}+\alpha_{k_{i-1}+1} + \cdots +\alpha_{i-1}+\alpha_i ) = \alpha_i
\end{eqnarray*}
and $w_1, w_2, ..., w_{i-2}$ stabilize $\alpha_i$.  Putting this together, we conclude $\ell(w_i) > \ell(w_{i-1})$ if and only if $\alpha_i \in N(w)$
as desired.
\end{proof}

The previous lemma is the key step in the next proposition, which shows that if $v\in W^J$ indexes a permutation flag $vB\in \B^X$ then the corresponding Schubert point $v_{T}$ is also in $W^J$.

\begin{prop}\label{shortest cosets and the algorithm}  Let $vB\in \B^X$.  Then $v\in W^J$ if and only if $v_{T}\in W^J$.
\end{prop}
\begin{proof}  Let $T$ denote the row-strict tableau associated to $v$.  We decompose $v_{T}$ into  $i$-strings as $v_{T}=v_{n-1} v_{n-2} \cdots v_2 v_{1}$.  Throughout this proof, assume $i$ satisfies $1\leq i\leq n-1$ and $\alpha_i\in J$.

By definition $\ell(v_i)=\ell_{i}$ and $\ell(v_{i-1})=\ell_{i-1}$ so by Proposition \ref{shortest coset characterization} and Remark \ref{fact: shortest coset representatives} we have only to show that $\alpha_i \notin N(v)$ if and only if $\ell_{i}\leq \ell_{i-1}$.  First $\alpha_i\notin N(v)$ if and only if $v(i)<v(i+1)$ by definition of inversions.  Since $i$ fills the box labeled by $v(i)$ in the base filling of $\lambda$, the inequality $v(i)<v(i+1)$ holds if and only if $i$ occurs in a box of $T$
\begin{itemize}
\item in the same column and below $i+1$, or
\item in a column to the left of $i+1$.
\end{itemize}   
Now consider $T[i]$ and $T[i+1]$.  We obtain $T[i]$ from $T[i+1]$ by removing the box containing $i+1$.  Lemma \ref{counting rows} states that $\ell_{i}$ counts the number of rows in $T[i+1]$ above the row containing $i+1$ and of equal length plus the total number of rows in $T[i+1]$ of length strictly greater than the row with $i+1$.  These rows each have the same length in $T[i]$ since they do not contain $i+1$; denote the set of rows by $\mathcal{R}$.  If $i$ satisfies either bulleted condition above then each row in $\mathcal{R}$ contributes one $i$-row inversion of $T$ to the count of $\ell_{i-1}$ so by Lemma \ref{counting rows} we have $\ell_{i}=|\mathcal{R}|\leq \ell_{i-1}$.  Conversely if $i$ satisfies neither bulleted condition then $\ell_{i-1}$ counts only a subset of $\mathcal{R}$  since $\mathcal{R}$ includes the row containing $i$.  Therefore $\ell_{i-1}<|\mathcal{R}|=\ell_{i}$. This proves the claim. 
\end{proof}

\begin{cor}\label{cor.lowerideal}  Suppose $X$ corresponds to a partition with at most three rows or two columns.  Then the set $\{ v_T\in W^J\mid v\in W^J \textup{  and  } vB\in \B^X \}$ is a lower order ideal with respect to Bruhat order on $W^J$.  In other words if $v'\in W^J$ and $v'\leq v_{T}$ for some $v_{T}$ in the set, then $v'$ is also an element of the set.  
\end{cor}
\begin{proof}  To prove this, we show that for each $v'\in W^J$ such that $v'\leq v_T$ there exists $y\in W^J$ with $yB \in \B^X$ and row-strict tableau $T'$ such that $v'=y_{T'}$.  By Proposition \ref{schubertpoint}, there exists a unique $yB\in \B^X$ and corresponding row-strict tableau $T'$ such that $v'=y_{T'}$.  By Proposition \ref{shortest cosets and the algorithm} this $y$ must also be an element of $W^J$ since $y_{T'}$ is.
\end{proof}

\begin{rem}
It's also important to note what this corollary does not say: this set is a lower order ideal in $W^J$ but {\em not necessarily} in $W$.  The next example shows how this can happen.
\end{rem}

\begin{example}\label{example: parabolic union}
Continue our example when $\lambda=\mu=(2,2)$.  Example~\ref{example: factorization formula} gave the set $W(X,J_{(2,2)}) = \{e,s_2,s_1s_3s_2\}$.  Example~\ref{ex.(2,2)contd} listed the row-strict tableaux corresponding to the elements in $W(X,J_{(2,2)})$.  The permutation $s_1s_3s_2$ corresponds to $T= \young(12,34)$ and Example~\ref{ex.(2,2)contd} explained that $\ell_3=\ell_2=1$ were the only nonzero contributions to the dimension.  By definition we obtain $v_{T} = s_3s_2$.  Similarly the row-strict tableau corresponding to $s_2$ is $T'=\young(34,12)$ with $v_{T'}=s_2$ and $e$ corresponds to the base filling, so $\{v_T: v \in W(X,J)\} = \{s_3s_2,s_2,e\}$ in this case.  Note that $s_3$ is not in this set, though $s_3 < s_3s_2$ in Bruhat order.  This is because $s_3 \not \in W^J$. 
\end{example}

Corollary~\ref{cor.lowerideal} immediately implies that the Poincar\'{e} polynomial of the Steinberg variety agrees with that of a union of Schubert varieties \textit{in the partial flag variety}.

\begin{cor}\label{corollary: steinberg schuberts}
Suppose $X\in \mathfrak{gl}_n(\C)$ is nilpotent with Jordan form corresponding to a partition $\lambda$ with at most three rows or two columns.  Then the following Poincar\'e polynomials are equal:
\[
\mathsf{P}(\pi_J(\B^X), t) = \mathsf{P}\left( \cup_{v\in W(X,J)} \overline{C}_{v_T}^{P_J}, t \right)\]
where $\overline{C}_{v_T}^{P_J}$ is a Schubert variety in the partial flag variety $G/P_J$.
\end{cor}

\begin{proof}
Corollary~\ref{cor.paving} tells us that the Steinberg variety is paved by the cells $C_v^{P_J}\cap\pi_J(\B^X)$ for $v \in W(X,J)$ and that $\dim(C_v^{P_J} \cap \pi_J(\B^X)) = \dim(C_v\cap \B^X)$ 
for each of these cells.  In addition $\dim(C_v\cap \B^X) = \ell(v_T)$ by construction.  Corollary~\ref{cor.lowerideal} now tells us that $\{v^T \in W^J: v \in W(X,J)\}$ is a lower order ideal.  Since $W^J$ indexes the permutation flags in $G/P_J$ this means the union of Schubert varieties $\overline{C}_{v_T}^{P_J}$ in the partial flag variety $G/P_J$ has the same Poincar\'{e} polynomial as the Steinberg variety, as desired.
\end{proof}

\begin{example}
Continuing our running example, Example~\ref{example: factorization formula} showed that when $\lambda=\mu=(2,2)$ the Poincar\'{e} polynomial of the Steinberg variety $\pi_J(\B^X)$ is $1+t+t^2$.  This is also the Poincar\'{e} polynomial of the Schubert variety $\overline{C}_{s_3s_2}^{P_J}$ in $G/P_J$.  \textup{(}In contrast, the Poincar\'{e} polynomial of the Schubert variety $\overline{C}_{s_3s_2} \subseteq G/B$ is $1+2t+t^2$.\textup{)}
\end{example}

We are now ready to state and prove the main theorem of this section. 

\begin{thm}\label{Main Theorem}  Suppose $X\in \mathfrak{gl}_n(\C)$ is nilpotent with Jordan form corresponding to a partition $\lambda$ with at most three rows or two columns.  Then the following Poincar\'e polynomials are equal:
\[
\mathsf{P}(\B(X,\fp_J), t) = \mathsf{P} \left(\cup_{v \in W(X,J)} \overline{C}_{v_{T}w_J}, t\right) 
\]
where $w_{J}$ denotes the longest word in $W_J$.
\end{thm}
\begin{proof} Note that the union of Schubert varieties is the disjoint union of Schubert cells
\[
\bigcup_{v \in W(X,J)} \overline{C}_{v_{T}w_J}												
= \bigsqcup_{v \in W(X,J) } \bigsqcup_{y\in W_J} C_{v_{T} y}
\]
because Schubert points are distinct and because $W(X,J)$ is a subset of coset representatives for $W/W_J$.  Recall that $M_J/B_J$ denotes the flag variety $M_J/(B\cap M_J)$ of $M_J$ and in particular that $\mathsf{P}(M_J/B_J,t) = \sum_{y \in W_J}t^{\ell(y)}$.  Thus we have
\begin{eqnarray*}
\mathsf{P}(\cup_{v  \in W(X,J)} \overline{C}_{v_{T}w_J} , t ) 										
&=& \sum_{v \in W(X,J)} t^{\ell(v_{T})} \mathsf{P}(M_J/B_J,t)\\
&=& \sum_{v \in W(X,J)} t^{ \dim(C_v\cap \B^X) } \mathsf{P}(M_J/B_J,t)\\
&=& \mathsf{P}(\B(X,\fp_J), t)
\end{eqnarray*}
where the last two equalities follow from Definition~\ref{defn: Schubert point} and Corollary~\ref{Poincare polynomial}, respectively.
\end{proof}

\begin{example}
Example~\ref{example: factorization formula} studied the parabolic Hessenberg variety when $X$ is nilpotent of Jordan type $\lambda= (2,2)$ and $J$ corresponds to the partition $(2,2)$ and found its Poincar\'{e} polynomial:
\[\mathsf{P}(\B(X,\fp_J), t) =  (1+t+t^2)(1+2t+t^2)= 1+3t+4t^2+3t^3+t^4.\]
This is precisely the  Poincar\'e polynomial of the Schubert variety $\overline{C}_{s_3s_2s_3s_1}$ computed in Example~\ref{Schubertex}.  
\end{example}


\section{Components of parabolic Hessenberg varieties} \label{Conclusion}

One natural follow-up question is whether the combinatorial results of Proposition~\ref{prop.semistandard}, Corollary~\ref{corollary: steinberg schuberts}, and Theorem~\ref{Main Theorem} reflect an underlying geometric property.  We now give one result in this direction, proving that the irreducible components of parabolic Hessenberg varieties are in bijection with the irreducible components of a Steinberg variety.  The following is the main result of this section, and holds in all Lie types.

\begin{thm}\label{thm.components} Fix $X\in \fb$.  Let $\pi_J: G/B \to G/P_J$ be the projection $\pi_J (gB) = gP_J$.  Under this map, the irreducible components of parabolic Hessenberg variety $\B(X,\fp_J)$ are in bijection with those of the Steinberg variety $\pi_J(\B^X)$.
\end{thm}

\begin{proof} Let $\B(X,\fp_J) = \cup_{i\in I}\cx_i$ be the decomposition of $\B(X,\fp_J)$ into irreducible components.  The map $\pi_J$ is continuous so each $\pi_J(\cx_i)$ is irreducible. Theorem~\ref{pullback} showed that $\pi_J(\B(X,\fp_J)) = \pi_J(\B^X)$  so $\pi_J(\B^X)$ can be written as a union $\cup_{i\in I} \pi_J(\cx_i)$.  To show that each $\pi_J(\cx_i)$ is a component, we prove that if $\pi_J(\cx_i)\subseteq \pi_J(\cx_j)$ then $i=j$. If $\pi_J(\cx_i)\subseteq \pi_J(\cx_j)$ then naturally $\pi_J^{-1}\pi_J(\cx_i)\subseteq \pi_J^{-1}\pi_J(\cx_j)$.  Thus it suffices to show that $\pi_J^{-1}\pi_J(\cx_i) = \cx_i$ since the $\cx_i$ are by definition components.

Suppose $g_1B\in \pi_J^{-1}(\pi_J(\cx_i))$.  Since $\pi_J(g_1B) \in \pi_J(\cx_i)$ there exists $g_2B \in \cx_i$ with $\pi_J(g_1B)=\pi_J(g_2B)$.  By statements (2) and (3) of Lemma~\ref{proj properties} we can write $g_1 = uvu_1y_1$ and $g_2=uvu_2y_2$ where $v \in W^J$, $y_1$ and $y_2$ are both in $W_J$, and $u\in U^v$, $u_1 \in U^{y_1}$, $u_2 \in U^{y_2}$.

Let $\mathcal{Z} = \{uvmB\mid m\in M_J\}\subseteq \B(X, \fp_J)$.  Then $g_1B, g_2B\in \mathcal{Z}$ and $\mathcal{Z}$ is isomorphic to the flag variety $M_J/B_J$.  Therefore $\mathcal{Z}$ is an irreducible subvariety of $\B(X,\fp_J)$, and must be contained in a single irreducible component of $\B(X,\fp_J)$.  This implies $\mathcal{Z}\subseteq \mathcal{X}_i$ so $g_1B\in \cx_i$ as desired. 
\end{proof}

As an immediate corollary, we conclude that in type $A$, the number of irreducible components of $\B(X,\fp_J)$ with dimension $\dim(\B^X) + \ell(w_J)$ is the Kostka number $K_{\lambda \mu}$.  The proof just applies Corollary~\ref{cor.irredcomp}, namely Steinberg's result on $\pi_J(\B^X)$.  

\begin{cor}\label{cor.componentsA} Let $\lambda$ and $\mu$ be partitions of $n$, $X\in \mathfrak{gl}_n(\C)$ be a nilpotent matrix with Jordan form determined by $\lambda$, and $J=J_\mu$.  The number of irreducible components of $\B(X,\fp_J)$ of dimension $\dim(\B^X) + \ell(w_J)$ equals the Kostka number $K_{\lambda\mu}$.
\end{cor}

Corollary~\ref{cor.componentsA} tells us that some of the irreducible components of parabolic Hessenberg varieties are indexed by certain standard tableaux, specifically, the standard tableaux that become semistandard under the degeneration map.  However, this description does not characterize all irreducible components, as the following example demonstrates.


\begin{example}\label{irreducible}
Let $X\in \mathfrak{gl}_4(\C)$ be a nilpotent matrix of Jordan type $\lambda=(2,1,1)$ so $\dim(\B^X) = 3$.  Let $\mu=(2,2)$ so $J=J_{\mu} = \{ \alpha_1,\alpha_3 \}$ and $w_J=s_1s_3$.  Note that $K_{\lambda \mu}=0$ in this case, meaning $\dim(\B(X,\fp_J))< \dim(\B^X)+\ell(w_J) = 5$ by Corollary~\ref{cor.componentsA}.   Taking $X$ as in Definition~\ref{defn: highest form} we obtain $\Phi_X=\{ \alpha_3 \}$ and  
\[
W(X,J)=\{ e, s_2,  s_1s_2, s_2s_1s_3s_2  \}.
\]
Consider the points $v_1=s_1s_2$ and $v_2=s_2s_1s_3s_2$.  The table below displays the corresponding elements of $\RST(\lambda)$ and $\RSST(\lambda)$, and computes $v_Tw_J$ in each case.
\begin{center}
\begin{tabular}{c | c | c | c | c}
$v\in W(X,J)$ & $T\in \RST(\lambda)$ & $\phi_{\lambda, \mu}(T)\in \RSST(\lambda)$ & $v_T$ & $v_Tw_J$\\ \hline
 & \multirow{3}{*}{$\young(24,1,3)$} & \multirow{3}{*}{$\young(12,1,2)$} & \\
$v_1=s_1s_2$ & & & $s_1s_2$ & $s_1s_2s_1s_3$ \\
& & & & \\
& \multirow{3}{*}{$\young(12,4,3)$} &  \multirow{3}{*}{$\young(11,2,2)$} &  \\
$v_2=s_2s_1s_3s_2$ & & & $s_3s_2$ & $s_3s_2s_1s_3$ \\
&&&&
\end{tabular}
\end{center}
We claim that $\overline{C_{v_1w_J}\cap \B(X,\fp_J)}$ and $\overline{C_{v_2w_J} \cap \B(X,\fp_J)}$ are the irreducible components of $\B(X,\fp_J)$.   We know $\dim(C_{v_1w_J}\cap \B(X,\fp_J))= \ell(v_T)+\ell(w_J)=4$ from our analysis of parabolic Hessenberg varieties.  This is the same as $\dim(C_{v_1w_J})$ so in fact $C_{v_1w_J} \subseteq \B(X,\fp_J)$.  Thus
\[
\overline{C_{v_1w_J}\cap \B(X,\fp_J)} = \overline{C}_{v_1w_J}= \bigsqcup_{w\leq v_1w_J} C_w =   \bigsqcup_{w\leq v_1w_J} C_w\cap\B(X,\fp_J). 
\]
Since $v_2w_J \nleq v_1w_J$ and the Hessenberg Schubert cells corresponding to $v_2w_J$ and $v_1w_J$ have the same dimension, neither of $\overline{C_{v_1w_J}\cap \B(X,\fp_J)}$ and $\overline{C_{v_2w_J} \cap \B(X,\fp_J)}$ can contain the other.  Since $vw_J \leq v_1w_J$ for all other $v\in W(X,J)$, we conclude
\[
	\B(X,\fp_J) = \overline{C}_{v_1w_J} \cup (\overline{C_{v_2w_J} \cap \B(X,\fp_J)}).
\] 
In particular, note that neither irreducible component corresponds to a standard \textup{(}or semistandard\textup{)} tableau of shape $\lambda$.  
\end{example}

Our partial description of the irreducible components of $\B(X,\fp_J)$ leads to the following question.

\begin{question}\label{question}  Suppose $\pi_J(\B^X)$ is paved by Steinberg Schubert cells.  Give a combinatorial description of those $v\in W(X,J)$ for which $\overline{C_v^{P_J}\cap \pi_J(\B^X)}$ is an irreducible component of the Steinberg variety.
\end{question}

Any answer to this question would also compute the irreducible components of the corresponding parabolic Hessenberg variety.  Motivated by Example \ref{irreducible}, one possibility is that $\overline{C_{v}\cap \pi_J(\B^X)}$ is an irreducible component of $\pi_J(\B^X)$ if the Schubert point $v_T$ corresponding to $v$ is a maximal in the set $\{v_T \mid v\in W(X,J)\}$. We have not been able to find a counterexample to this conjecture, but suspect that there is one. 

In addition, an answer to Question~\ref{question} would extend the known characterization of components of the Springer fibers in type $A$.  It appears, too, to require a deep analysis of the set $W(X,J)$ as well as its connection to the geometry of the Steinberg variety.

\newcommand{\etalchar}[1]{$^{#1}$}

\end{document}